\documentclass[11pt]{amsart}
\usepackage[a4paper,left=3cm,right=3cm,bottom=2.5cm]{geometry}
\usepackage[utf8]{inputenc}
\usepackage[english]{babel}
\usepackage{amsmath,amssymb}
\usepackage{amsthm}
\usepackage{tikz}
\usepackage{listings}
\usepackage{color} 
\definecolor{mygreen}{RGB}{28,172,0} 
\definecolor{mylilas}{RGB}{170,55,241}

\newtheorem{theorem}{Theorem}[section]
\newtheorem{conjecture}{Conjecture}

\newtheorem{corollary}{Corollary}[theorem]
\newtheorem{lemma}[theorem]{Lemma}
\newtheorem{proposition}[theorem]{Proposition}
\newtheorem{definition}[theorem]{Definition}
\newtheorem{remark}[theorem]{Remark}

\newcommand{\R}{\mathbb R}
\newcommand{\eps}{\varepsilon}

\begin{document}
\title{The diagram $(\lambda_1,\mu_1)$}
\author[]{Ilias Ftouhi, Antoine Henrot}

\address[Ilias Ftouhi]{Friedrich-Alexander-Universität  Erlangen-Nürnberg, Department of Mathematics, Chair in Applied Analysis – Alexander von Humboldt Professorship, Cauerstr. 11, 91058 Erlangen, Germany.}
\email{ilias.ftouhi@fau.de}
\address[Antoine Henrot]{Universit\'e de Lorraine, CNRS, IECL, F-54000 Nancy, France}
\email{antoine.henrot@univ-lorraine.fr}

\lstset{language=Matlab,%
    breaklines=true,%
    morekeywords={matlab2tikz},
    keywordstyle=\color{blue},%
    morekeywords=[2]{1}, keywordstyle=[2]{\color{black}},
    identifierstyle=\color{black},%
    stringstyle=\color{mylilas},
    commentstyle=\color{mygreen},%
    showstringspaces=false,
    numbers=left,%
    numberstyle={\tiny \color{black}},
    numbersep=9pt, 
    emph=[1]{for,end,break},emphstyle=[1]\color{red}, 
}

\date{September 30, 2021}

\begin{abstract}
 In this paper we are interested in the possible values taken by the pair $(\lambda_1(\Omega), \mu_1(\Omega))$ the first eigenvalues of the Laplace
 operator with Dirichlet and Neumann boundary conditions respectively of a bounded plane domain $\Omega$. We prove that, without any particular
 assumption on the class of open sets $\Omega$, the two classical inequalities (the Faber-Krahn inequality and the Weinberger inequality) provide
 a complete system of inequalities. Then we consider the case of convex plane domains for which we give new inequalities for the product $\lambda_1 \mu_1$.
 We plot the so-called Blaschke--Santal\'o diagram and give some conjectures.
\end{abstract}



\maketitle

{\it Dedicated to Marius Tucsnak for his 60th birthday}

\section{Introduction}

Let $\Omega\subset \mathbb{R}^2$ be an open and bounded set in the plane and let us denote by $\lambda_1(\Omega)$ its first eigenvalue
for the Laplacian-Dirichlet (see below for the precise definition) and $\mu_1(\Omega)$ its first non-trivial eigenvalue for the Laplacian-Neumann.
Does there exist a domain $\Omega$ of area, say $\pi$ satisfying $\lambda_1(\Omega)=20$ and $\mu_1(\Omega)=3$ ?
Among them, is there a convex domain?
For that kind of question, it is very convenient to plot the so-called {\it Blaschke--Santal\'o diagrams} defined by
\begin{equation}\label{diagE}
\mathcal{E}=\{(x,y) \mbox{ where } x=|\Omega| \lambda_1(\Omega),\; y=|\Omega| \mu_1(\Omega),\; \Omega\in \mathcal{O}\},
\end{equation}
and
\begin{equation}\label{diagEC}
\mathcal{E}^C=\{(x,y) \mbox{ where } x=|\Omega| \lambda_1(\Omega),\; y=|\Omega| \mu_1(\Omega),\; \Omega\in \mathcal{K}\}
\end{equation}
where $\mathcal{O}$ is the class of open bounded subsets of Lipschitz boundary of $\mathbb{R}^2$ and $\mathcal{K}$ is the class of open bounded convex sets of $\mathbb{R}^2$ with non-empty interior. These diagrams describe all the possible values of the couple ($\lambda_1,\mu_1$). This kind of diagram has been s{introduced} by W. Blaschke
in convex geometry and intensively studied by L. Santal\'o for quantities like the area, the perimeter, the diameter, the inradius...
For spectral quantities, this kind of work is more recent, let us mention for example some
diagrams that have been recently established for quantities like $(\lambda_1(\Omega),\lambda_2(\Omega))$ (the Dirichlet
eigenvalues) in \cite{AH11}, \cite{BBFi}, $(\mu_1(\Omega),\mu_2(\Omega))$ (the Neumann eigenvalues) in \cite{AH11},  
$(\lambda_1(\Omega),T(\Omega))$ {(where $T(\Omega)$ is the torsion of $\Omega$ \footnote{The torsion function of the set $\Omega$ can be defined by $T(\Omega):=\sup\limits_{w\in H^1_0(\Omega)} \frac{\left(\int_\Omega w d x\right)^2}{\int_\Omega |\nabla w|^2 d x}$.})} in 
\cite{BBP19}, \cite{MR4251321} or \cite{LZ19},  $(P(\Omega),\lambda_1(\Omega))$ (here
$P(\Omega)$ is the perimeter) in \cite{FL21}.

Let us now fix the notations:
the Laplace-Dirichlet problem on $\Omega$ consists in solving the eigenvalue problem 
\begin{equation*}
\begin{cases}
    - \Delta u&=\ \ \lambda u \qquad\mbox{in }  \Omega   \\
      \ \ \ \ \ u &=\ \ 0 \ \ \qquad\mbox{on } \partial\Omega.
\end{cases}
\end{equation*}
For $\Omega$ {open and} bounded,  the spectrum  is discrete and the sequence of eigenvalues (counted with their multiplicities) go to infinity
$$0< \lambda_1(\Omega)\le  \lambda_2 (\Omega) \le \cdots \rightarrow +\infty.$$

The Laplace-Neumann eigenvalue problem on $\Omega$ consists in solving the eigenvalue problem 
\begin{equation*}
\begin{cases}
     -\Delta u&=\ \ \mu u\qquad\mbox{in }   \Omega \\
     \ \  \partial_{\nu} u&=\ \ 0 \ \ \qquad\mbox{on }   \partial\Omega.
\end{cases}
\end{equation*}
where $\nu$ stands for the outward unit normal at the boundary.
We assume here $\Omega$ to be a Lipschitz bounded open set. Since the Sobolev embedding $H^1(\Omega) \rightarrow L^2(\Omega)$ is compact in that case, 
the spectrum of the Neumann problem is discrete and the eigenvalues (counted with their multiplicities) go to infinity. The first eigenvalue is zero, associated to constant functions.
$$0= \mu_0(\Omega) \leq   \mu_1(\Omega)\le  \mu_2 (\Omega) \le \cdots \rightarrow +\infty.$$
We denote by $\mu_1(\Omega)$ the next or first non-trivial eigenvalue. Note that $\mu_1(\Omega)=0$ if and only if $\Omega$ is disconnected.
We recall the variational characterization of the Neumann eigenvalues, that for the first non zero eigenvalue reads as follows
\begin{equation}\label{eVFN}
\mu_1(\Omega)=\min \Big \{  \frac{\int_{\Omega}|\nabla u|^2dx}{\int_{\Omega}u^2 dx}:u\in H^1(\Omega),\int_{\Omega}udx=0 \Big \},
\end{equation}
and the minimum is attained at the eigenfunctions associated to $\mu_1(\Omega)$.

\medskip\noindent

In Section \ref{secgen}, we describe the diagram $\mathcal{E}$ defined in \eqref{diagE}. It turns out that it is (almost) completely characterized by the two classical inequalities:
the Faber-Krahn inequality, see \cite{Faber1923}, \cite{Krahn}:
\begin{equation}\label{FKi}
\forall \Omega {\in\mathcal{O}} ,\quad |\Omega| \lambda_1(\Omega) \geq |B| \lambda_1(B)
\end{equation}
where $B$ is any ball
and the Weinberger inequality, see \cite{W56}, \cite{Sz54}
\begin{equation}\label{Wi}
\forall \Omega\ {\in\mathcal{O}},\quad |\Omega| \mu_1(\Omega) \leq |B| \mu_1(B),
\end{equation}
where we define $\mathcal{O}$ to be the class of open bounded subsets of Lipschitz boundary of $\mathbb{R}^2$.

When we say completely characterized, we mean that it is a complete set of inequalities: in other words, the closure of the diagram $\mathcal{E}$ coincides actually with the
strip $[|B| \lambda_1(B), +\infty) \times [0,|B| \mu_1(B)]$.
We will denote by $A= (|B| \lambda_1(B),|B| \mu_1(B))=(\pi j_{0,1}^2,\pi {j'}_{11}^2)$ the upper left vertex corresponding to the ball,
where $j_{0,1}\sim 2.405$ is the first zero of the Bessel function $J_0$ while $j'_{11}\sim 1.841$ is the first zero of the derivative of the Bessel function $J_1$,
see e.g. \cite{He06}.

\medskip\noindent
Then we study the diagram  $\mathcal{E^C}$ defined in \eqref{diagEC}. It is more complicated, but we give two explicit curves (hyperbola) that bound the diagram from
above and from below.
In particular, coming back to the question raised at the beginning of this Introduction, we can answer: 

\begin{itemize}
\item yes, it could exist a plane domain of area $\pi$ satisfying $\lambda_1(\Omega)=20$ and $\mu_1(\Omega)=3$. Indeed, the point $(20,3)$ belongs to
the strip $[|B| \lambda_1(B), +\infty) \times [0,|B| \mu_1(B)]$, but since we are only able to prove that this strip is the closure of the set $\mathcal{E}$, we cannot
claim that it corresponds to a set $\Omega$.
\item no, there does not exist a convex domain satisfying this.
\end{itemize}

\section{The diagram for the general case}\label{secgen}
\subsection{Convergence of eigenvalues}
It will be useful in the sequel to have simple criteria ensuring convergence of the Neumann and Dirichlet eigenvalues for a sequence of domains.
In the Dirichlet case, the situation is well understood and we can state for example the following theorem, see \cite{HPb} or \cite{He06}.
\begin{theorem}[Sverak]\label{thmsv}
Let $\Omega_n \subset \R^2$ be a sequence of bounded open sets converging for the Hausdorff metric to an open set $\Omega$. Assume that the
number of connected components of the complement {set} of $\Omega_n$ is uniformly bounded, then the Dirichlet eigenvalues converge:
$$\forall k,\quad \lambda_k(\Omega_n) \to \lambda_k(\Omega).$$
\end{theorem}
In the Neumann case, we need to put much stronger assumption on the sequence of domains $\Omega_n$. A classical one in that context is {\it the uniform cone condition},
see \cite{Agb}, \cite{Che75}, \cite{HPb}:
\begin{definition}\label{deficone}
Let $y$ be a point in $\mathbb{R}^N$, $\xi$ a unit vector and
$\varepsilon$ a positive real number. Let $C(y,\xi,\varepsilon)$ be
the cone of vertex $y$ (without its vertex), of direction $\xi$ and dimension $\varepsilon$
defined by 
$$C(y,\xi,\varepsilon)=\{z\in\mathbb{R}^N,\ (z-y,\xi)\geq
\cos(\varepsilon)|z-y|\quad {\rm and}\quad 0<|z-y|<\varepsilon\}.$$
An open set $\Omega$ is said to have the uniform ($\varepsilon$)-cone property
if
$$\forall x\in\partial\Omega,\ \exists\xi_x\ {\rm
unit\;vector\;such\;that:}\quad\forall y\in \overline{\Omega}\cap
B(x,\varepsilon)\quad C(y,\xi_x,\varepsilon)\subset\Omega.$$
\end{definition}
Now we work with a sequence of sets $\Omega_n$ which all have the uniform ($\varepsilon$)-cone property {\bf for the same $\varepsilon$}.
An equivalent definition is to say that the domains $\Omega_n$ are all uniformly Lipschitz, with the same Lipschitz constant.
Let $D$ be a ball containing all the sets $\Omega_n$ (that can be seen as an assumption). It is well known that the cone property (or the Lipschitz regularity)
ensures an existence of an extension operator $P_n:H^1(\Omega_n) \to H^1(D)$
The important point proved by D. Chenais in \cite{Che75} is the following:
\begin{lemma}\label{lemuniproj}
If the sets $\Omega_n$ have the uniform ($\varepsilon$)-cone property (for the same $\varepsilon$), then there exists a constant $M$ such that
$\|P_n\|\leq M$.
\end{lemma}
We can deduce the following convergence theorem for the first Neumann eigenvalue:
\begin{theorem}\label{theneum}
Let $\Omega_n\subset D$ be a sequence of open sets having the uniform $\varepsilon$-cone property (for the same $\varepsilon$). Let us
assume that there exists an open set $\Omega$ such that $\chi_{\Omega_n}$ (the sequence of corresponding characteristic functions) converges
in $L^1(D)$ to $\chi_\Omega$. Then
$$\mu_1(\Omega_n) \to \mu_1(\Omega).$$
\end{theorem}
Actually, the convergence holds true for all eigenvalues, but here we only need convergence of $\mu_1$, thus we state it in that case.
\begin{proof}
Since $\chi_{\Omega_n} \to \chi_\Omega$, we have $|\Omega_n|\to |\Omega|$ and then, we can assume $|\Omega_n|\geq |\Omega|/2$.
Let us observe first that, according to Weinberger inequality \eqref{Wi}, the sequence $\mu_1(\Omega_n)$ is bounded by a constant $M_1=2\pi {j'_{11}}^2/|\Omega|$.
Thus, we can assume that, up to a subsequence that we still index by $n$, we have $\mu_1(\Omega_n)$ converges to some number $\mu$.
Let us denote by $u_n$ the normalized eigenfunction associated to $\mu_1(\Omega_n)$ i.e.  
$$\int_{\Omega_n} u_n^2 dx=1, \quad \int_{\Omega_n} |\nabla u_n|^2 dx= \mu_1(\Omega_n),\quad \int_{\Omega_n} u_n dx=0$$
and $\tilde{u}_n=P_n(u_n)$ its extension to $D$ through the extension operator $P_n$. 

From Lemma \ref{lemuniproj}, we get 
$$\|\tilde{u}_n\|_{H^1(D)} \leq M\|u_n\|_{H^1(\Omega_n)}=M\sqrt{1+\mu_1(\Omega_n)}\leq M\sqrt{1+M_1}.$$
Therefore, there exists a function $u_\infty\in H^1(D)$ and a subsequence, still denoted with the same index, such that $\tilde{u}_n$ converges strongly in $L^2(D)$
and weakly in $H^1(D)$ to $u_\infty$. Let us consider a fixed $v\in H^1(D)$, passing to the limit in the variational formulation 
$$\int_{\Omega_n} \nabla u_n.\nabla v dx=\int_D \chi_{\Omega_n} \nabla \tilde{u}_n.\nabla v dx = \mu_1(\Omega_n) \int_D \chi_{\Omega_n} \tilde{u}_n v dx$$
yields
$$\int_\Omega \nabla u_\infty.\nabla v dx = \mu \int_\Omega u_\infty v dx.$$
Moreover, passing to the limit in the same way provides the following identities
$$\int_\Omega u_\infty dx =0 \qquad \int_\Omega u_\infty^2 dx =1$$
this shows that $u_\infty$ is neither zero, nor constant. It follows that $\mu$ is a nontrivial Neumann eigenfunction of $\Omega$
and therefore
\begin{equation}\label{limi1}
\mu_1(\Omega) \leq \liminf \mu_1(\Omega_n).
\end{equation}
On the other hand, let $u$ be the normalized eigenfunction associated to $\mu_1(\Omega)$, $P(u)$ its extension to $D$ and let us consider $v_n$ the restriction
of $P(u)$ to $\Omega_n$ (which is a function in $H^1(\Omega_n)$). We denote by $M_n$ its mean value defined as $\frac{1}{|\Omega_n|} \int_{\Omega_n} v_n dx$.
We observe that we have the following convergences:
$$M_n \to \frac{1}{|\Omega|} \int_{\Omega} u dx =0 ,\quad \int_{\Omega_n} v_n^2 dx \to  \int_{\Omega} u^2 dx=1$$
and 
$$\int_{\Omega_n} |\nabla v_n|^2 dx \to  \int_{\Omega} |\nabla P(u)|^2 dx=\int_{\Omega} |\nabla u|^2 dx=\mu_1(\Omega).$$
Thus taking $v_n-M_n$ as a test function in the variational formulation of $\mu_1(\Omega_n)$ and passing to the limit provides
$$\limsup \mu_1(\Omega_n) \leq \limsup \dfrac{\int_{\Omega_n} |\nabla v_n|^2 dx}{\int_{\Omega_n}  (v_n - M_n)^2 dx} = \mu_1(\Omega)$$
together with \eqref{limi1} this provides the expected continuity.
\end{proof}
\subsection{The diagram $\mathcal{E}$}
We recall that the diagram we want to plot is
$$\mathcal{E}=\{(x,y) \mbox{ where } x=|\Omega| \lambda_1(\Omega),\; y=|\Omega| \mu_1(\Omega),\; \Omega\subset \R^2\}$$
where $\Omega$ is any bounded, open set suitably smooth (e.g. Lipschitz) in order that the Neumann spectrum is well defined.
Faber-Krahn \eqref{FKi} and Weinberger \eqref{Wi} inequalities imply that the upper left vertex of the diagram is the point 
$$A=(|B| \lambda_1(B),|B| \mu_1(B))=(\pi j_{0,1}^2,\pi {j'_{1,1}}^2).$$
We begin with the following lemma that is based on homogenization theory:
\begin{lemma}\label{lemhom}
Let $(x_0,y_0)\in \mathcal{E}$, then the half line $\{(x,y_0), x\geq x_0\}\subset \bar{\mathcal{E}}$, { where $\bar{\mathcal{E}}$ is the closure of $\mathcal{E}$}.  
\end{lemma}
\begin{proof}
For the proof, we recall the construction of a sequence of perforated domains introduced by Cioranescu-Murat in \cite{CM97}
(we give the statement in two dimensions, but it is true in any dimension with a different normalization).
Let  $C_0>0$ be fixed. For every $\eps>0$, consider the ball $T_\eps= B_{r_\eps}(0)$ with a radius given by $r_\eps=\exp(-C_0/\eps^{2})$.
Now we consider the perforated domain
\begin{equation}\label{oe}
\Omega _\eps= \Omega  \setminus \cup_{z\in \mathbb Z^2}(2\eps z + \overline{T}_\eps)\,.
\end{equation}
Note that the removed holes form a periodic set in the plane, with period $2\eps$. Now it is proved in \cite{CM97} that the torsion functions on the
domains $\Omega_\eps$ (that is the solution of $-\Delta u=1$ in $\Omega_\eps$ and $u=0$ on the boundary) converge weakly in $H^1_0(\Omega )$ 
(and strongly in $L^2(\Omega)$)
to the solution $u^*$ of
$$
\left\{\begin{array}{lll}
-\Delta u^* +  \frac{\pi}{2 C_0} u^* = 1 \  &  \hbox{in }\Omega
\\
u^*\in H^1_0(\Omega )\,.\end{array}.
\right.
$$
As a consequence we have (see \cite[Theorem 2.3.2]{He06})
$$\lambda_1(\Omega_\eps) \to \lambda_1(\Omega) +  \frac{\pi}{2 C_0}$$
while a simple computation of the total area of the holes provides $|\Omega_\eps| \to  |\Omega|$. Therefore, $x(\Omega_\eps) \to x_0+ \frac{\pi |\Omega|}{2 C_0}$
and $C_0$ being arbitrary, we can attain any value of $x$ greater than $x_0$.

On the other hand, the behavior of the Neumann eigenvalues is simpler: it can be proved, see \cite{Va81} that $\mu_1(\Omega_\eps) \to \mu_1(\Omega)$,
therefore $y(\Omega_\eps) \to y_0$, this concludes the proof of Lemma \ref{lemhom}.
\end{proof} 
Following Lemma \ref{lemhom}, we need now to have a precise look on the left of the diagram. For that purpose, we prove:
\begin{lemma}\label{lemleft}
Let $\eta >0$ be given, then there is a path in $\mathcal{E}$ connecting the points $A=(|B| \lambda_1(B) ,|B| \mu_1(B))$ corresponding to the ball
with the point $C=(|B| \lambda_1(B), 0)$ that is completely contained in the rectangle $[|B| \lambda_1(B),  |B| \lambda_1(B) + \eta] \times [0, |B| \mu_1(B)]$.
\end{lemma}
\begin{proof}
The proof is based on an explicit construction together with a sequence of domains {\it \`a la} Courant-Hilbert whose Neumann eigenvalue goes to zero
while its Dirichlet eigenvalue goes to the eigenvalue of the ball.

Let us start with the unit disc $\mathbb{D}$ and, for a given (small) number $\eps>0$, let us denote by $\Omega_\eps$ the union of the unit disc
with the rectangle { $(0,1+2\varepsilon)\times (-\eps^6/2,+\eps^6/2)$} and the small disc of center $(1+2\eps,0)$ and radius $\eps$, see {Figure \ref{fig:CH}}.

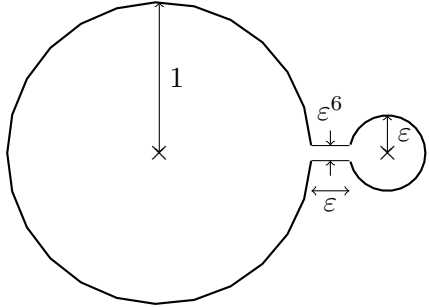
\begin{figure}[!h]
 \begin{tikzpicture}
   \draw (0,0) node {$\times$};
   \draw (3,0) node {$\times$};
   \draw [black,thick,domain=3:357] plot ({2*cos(\x)}, {2*sin(\x)});
   \draw[-] (2,.1) -- (2.5,.1);
   \draw[<->] (0,0) -- (0,2);
   \draw[<->] (3,0) -- (3,.5);
    \draw[-] (2,-.1) -- (2.5,-.1);
\draw (0,1) node[right] {$1$}; 
\draw (3,.25) node[right] {$\eps$}; 
    \draw[->] (2.25,-.3) -- (2.25,-.1);
    \draw[->] (2.25,.3) -- (2.25,.1);
\draw (2.25,.3) node[above] {$\eps^6$}; 
\draw [black,thick,domain=193:360+167] plot ({3+.5*cos(\x)}, {.5*sin(\x)});
\draw[<->] (2,-.5) -- (2.5,-.5);
\draw (2.25,-.9) node[above] {$\eps$};
\end{tikzpicture}
\caption{The domain $\Omega_\eps$}
\label{fig:CH}
\end{figure}

The area of $\Omega_\eps$ is less than $\pi(1+\eps^2)+\eps^7$ while, by monotonicity of Dirichlet eigenvalues with respect to inclusion
$\lambda_1(\Omega_\eps) \leq \lambda_1(\mathbb{D})$. Therefore, if we choose $\eps$ such that $\eps^2 \leq \eta/(2\pi \lambda_1(\mathbb{D}))$
we ensure $x(\Omega_\eps)\leq |B| \lambda_1(B) + \eta$.

Now, in a first step, we continuously deform the unit disc up to arriving on $\Omega_\eps$. We can assume that the sequence we make is increasing
(for inclusion) and we can also do it by preserving the {\it uniform $\varepsilon$-cone property} (see Definition \ref{deficone}) ensuring continuity of the 
Neumann eigenvalue and of the Dirichlet eigenvalue.  This continuous deformation makes a first continuous path
between the point $A$ and the point $A_\eps=(x(\Omega_\eps),y(\Omega_\eps))$ and by monotonicity, this path stays inside the rectangle 
$[|B| \lambda_1(B),  |B| \lambda_1(B) + \eta] \times [0, |B| \mu_1(B)]$.

In a second step, we make $\eps$ goes to zero. By Sverak's theorem \ref{thmsv}, we infer $\lambda_1(\Omega_\eps) \to \lambda_1(\mathbb{D})$ continuously
(and the same for the areas), so $x(\Omega_\eps) \to \pi \lambda_1(\mathbb{D})$ continuously. The Neumann eigenvalue $\mu_1(\Omega_\eps)$ also varies
continuously, according to  Theorem \ref{theneum} since for each fixed $\eps$, the family of domains satisfy a uniform cone condition in a neighborhood of 
$\Omega_\eps$. Now, let us consider, as a test function, the following continuous  function $v$ defined on $\Omega_\eps$ and depending only on the variable $x$:
$$v_\eps(x)= \left\lbrace\begin{array}{cc}
-1 & \mbox{in $\mathbb{D}$}\\
a_\eps (x-1) -1 & \mbox{in the channel}\\
\eps a_\eps -1 & \mbox{in the small disc of radius $\eps$}
\end{array}\right.$$
%
%
where we choose $a_\eps$ in order to have $\int_{\Omega_\eps} v_\eps dX =0$. A simple computation provides $a_\eps\sim 1/\eps^3$
when $\eps \to 0$. Moreover, 
$$\int_{\Omega_\eps} v_\eps^2(X) dX \geq \int_{\mathbb{D}} v_\eps^2(X) dx =\pi,$$
therefore, the variational characterization \eqref{eVFN} of $\mu_1$ yields
$$\mu_1(\Omega_\eps)\leq \frac{\eps^7 a_\eps^2}{\pi} \to 0 \quad \mbox{when $\eps$ goes to 0}.$$
This finishes the proof.
\end{proof}
From Lemmas \ref{lemhom} and \ref{lemleft}, we infer:
\begin{theorem}\label{theogenE}
The closure of the diagram $\mathcal{E}$ defined by
$$\mathcal{E}=\{(x,y) \mbox{ where } x=|\Omega| \lambda_1(\Omega),\; y=|\Omega| \mu_1(\Omega),\; \Omega\subset \R^2\}$$
is the strip $[|B|\lambda_1(B),+\infty)\times [0,|B| \mu_1(B)]$ (where $B$ is any disc).
\end{theorem}
\begin{remark}\label{rk2.8} It is not clear to see whether the set $\mathcal{E}$ is closed in the general case, see Theorem \ref{thm3.1} below for the convex case. 
In the case of purely Dirichlet eigenvalues (the diagram $(\lambda_1(\Omega),
\lambda_2(\Omega))$  in \cite{BBFi} the authors can prove the closeness of the diagram using arguments of weak $\gamma$-convergence that are not true
for Neumann eigenvalues. The difficulty comes mainly from the weird behavior of Neumann eigenvalues with respect to set variations.
\end{remark}
\begin{remark}
For sake of simplicity, we have stated all the previous results for plane domains. It is straightforward to extend these results in any dimension.
In particular Theorem \ref{theogenE} remains true in any dimension. We leave the details to the reader.
\end{remark}
\section{The diagram for the convex case}\label{secconv}
\subsection{Qualitative properties}
As we will see, the diagram $\mathcal{E}^C$ corresponding to plane convex subsets cannot be found explicitly, but we are going to give
some qualitative properties and bounds allowing to have a more precise idea of this diagram. We will also give some numerical experiments.  We
define $\mathcal{K}$ to be the class of open bounded convex sets of $\mathbb{R}^2$ with non-empty interior.

Let us start with a topological property (difficult to prove without the convexity assumption as mentioned in Remark \ref{rk2.8}).
\begin{theorem}\label{thm3.1}
The set $\mathcal{E}^C$ is closed.\\
In particular, the shape optimization problems 
$$\min\{\mu_1(\Omega), \Omega {\in \mathcal{K}}, \lambda_1(\Omega)=x_0, |\Omega|=A_0\}$$
{and}
$$\max\{\mu_1(\Omega), \Omega  {\in \mathcal{K}}, \lambda_1(\Omega)=x_0, |\Omega|=A_0\}$$
have a solution.
\end{theorem}
\begin{proof}
Let $(x_n,y_n)$ be a sequence of points in $\mathcal{E}^C$ corresponding to a sequence of (bounded) convex open sets $\Omega_n$.
We assume that $(x_n,y_n)$ converges to some point $(x_0,y_0)$.
By scale invariance of the coordinates $x$ and $y$ (namely $x$ and $y$ are invariant by homothety: $x(t\Omega)=x(\Omega)$), 
we can assume without loss of generality, that the diameter of $\Omega_n$ is fixed, equal to 1.
Then the sequence $\Omega_n$ stays in a fixed ball of radius 2 and from the Blaschke selection theorem (or compactness property of the Hausdorff
convergence, see \cite[Theorem 2.2.25]{HPb}) we know that, up to a subsequence, $\Omega_n$ converges (for the Hausdorff metric or for characteristic functions)
to some convex open set $\Omega$. Moreover, we will see below in Theorem \ref{theobounds} that for all $n$,  $x_ny_n >\pi^4/4$, therefore $y_0\not= 0$ and $\Omega$
is not the empty set.  We immediately deduce the following convergences:
\begin{itemize}
\item $|\Omega_n| \to |\Omega|$ by convergence of the characteristic functions,
\item $\lambda_1(\Omega_n) \to \lambda_1(\Omega)$ by Sverak Theorem \ref{thmsv}.
\end{itemize}
For the Neumann eigenvalue, let us consider a fixed compact ball $B$ included into the limit domain $\Omega$. By a classical stability result,
see \cite[Proposition 2.2.17]{HPb}, we know that $B\subset \Omega_n$ for $n$ large enough. Therefore, following \cite[Theorem 2.4.4]{HPb}, we see that all
the convex sets $\Omega_n$ satisfy a uniform $\eps$-cone property (with the same $\eps$ depending only on the radius of $B$). Using Theorem \ref{theneum}
we conclude that $\mu_1(\Omega_n) \to \mu_1(\Omega)$. Therefore $(x_0,y_0)=(|\Omega| \lambda_1(\Omega),|\Omega| \mu_1(\Omega)) \in \mathcal{E}^C$
proving the closeness of this set.

The result on the existence of minimizers or maximizers follow immediately from the closeness (by taking minimizing or maximizing sequences).
\end{proof}

\begin{remark}
It is interesting to note that numerical simulations of Section \ref{S:numerics} suggest that the diagram $\mathcal{E}^C$ is simply connected and even horizontally and vertically convex (which is a stronger statement): indeed, one can conjecture that the diagram $\mathcal{E}^C$ is exactly the set of points located between the curves of two continuous and strictly decreasing functions. 
These curves are defined as the images of the solutions of the minimum and the maximum problem presented at Theorem \ref{thm3.1}. 
The reader should be aware that proving such properties can be very challenging, especially when dealing with diagrams in the class of convex sets, we refer for example to \cite[Conjecture 5]{AH11} for the couple $(\lambda_1(\Omega),\lambda_2(\Omega))$ (the Dirichlet eigenvalues) and \cite[Open problem 2]{LZ19} for the couple $(\lambda_1(\Omega),T(\Omega))$ (where $T(\Omega) $ is the torsion). A strategy based on some perturbation lemmas and judicious choices of continuous paths constructed via Mikowski sums has been introduced in \cite{FL21}. Unfortunately, proving such perturbation lemmas seems to be quite challenging for the involved functionals $\lambda_1$ and $\mu_1$. As for the case of open sets, we point out the Open Problem 3 stated in \cite{BBP19} and very recently solved in \cite{MR4251321}. 
\end{remark}

\medskip
Now we want to bound the diagram $\mathcal{E}^C$ by two hyperbola. To that purpose,
we introduce and study the following scaling invariant functional 
$$F(\Omega):= x(\Omega) y(\Omega) =  |\Omega|^2\lambda_1(\Omega)\mu_1(\Omega).$$
By Theorem \ref{theogenE}, we have for general open sets $\inf F(\Omega)=0$ and $\sup F(\Omega)=+\infty$. Now, for planar convex sets, we have:
\begin{theorem}\label{theobounds}
$$\forall \Omega {\in \mathcal{K}},\ \ \ \ \ \frac{\pi^4}{4}< F(\Omega) < 9\pi^2j_{01}^2  $$
\end{theorem}
\begin{proof}
{\bf Lower bound}\\
We denote by $A(\Omega)=A$ the area of a convex domain,  $r(\Omega)=r$ its inradius and $D(\Omega)=D$ its diameter.
For the lower bound, we have by Hersch inequality \cite{Her60}: 
$$\lambda_1(\Omega)\ge \frac{\pi^2}{4}\times \frac{1}{r^2},$$
and by Payne-Weinberger inequality \cite{PW60}
$$\mu_1(\Omega) > \frac{\pi^2}{D^2},$$
then, we deduce $F(\Omega)\geq \frac{\pi^4 A^2}{4r^2 D^2}$. Now, to estimate the geometric quantity in the right-hand side,
we use a result, by M. Hernandez-Cifre and G. Salinas \cite{Hernandez01}, see also \cite{DHP21} that can be written:
$$A\geq r\sqrt{d^2-4r^2}+r^2\left(\pi -2\arccos{\left(\frac{2r}{d}\right)}\right).$$
This implies
\begin{align*}
F(\Omega)> \frac{\pi^4}{4}\times \left(\frac{A}{d r}\right)^2 &\ge \frac{\pi^4}{4}\times \left(\frac{r\sqrt{d^2-4r^2}+r^2\left(\pi -2\arccos{\left(\frac{2r}{d}\right)}\right)}{d r}\right)^2 \\
&= \frac{\pi^4}{4}\times \left(\sqrt{1-\left(\frac{2 r}{d}\right)^2}+\frac{r}{d}\left(\pi -2\arccos{\left(\frac{2r}{d}\right)}\right)\right)^2 \\
&\ge \min\limits_{x\in \big(0,\frac{1}{4}\big]} \frac{\pi^4}{4}\times  \left(\sqrt{1-x^2}+\frac{x}{2}\,(\pi-2\arccos{(x)})\right)^2 = \frac{\pi^4}{4}.
\end{align*}

{\bf Upper bound}\\ 
For the upper bound we first give a recent result obtained in \cite{HLL21} and we give the proof by sake of completeness:
\begin{proposition}
Let $\Omega$ be any smooth and bounded open set in the plane. Let us denote by $A$ its area and $w$ its minimal width
(the minimal distance between two parallel lines enclosing $\Omega$). Then
\begin{equation}\label{HLLineq}
\mu_1(\Omega) \leq \frac{\pi^2 w^2}{A^2},
\end{equation}
with equality for rectangles.
\end{proposition}
\begin{proof}
Let us consider the minimal strip enclosing $\Omega$. Without loss of generality, we can take this strip as horizontal limited by the two lines
$y=0$ and $y=w$. Let us now consider the family of rectangles $(a,a+A/w)\times (0,w)$ of area $A=|\Omega|$. Let us denote by $L=A/w$
the length of the rectangle. Let us now consider
a potential test function defined as
$$v_a(x,y)=\left\lbrace\begin{array}{cc}
-1 & \mbox{if } x\leq a\\
\cos\left(\frac{\pi}{L}\,(x-a+L) \right) & \mbox{if } a<x<a+L \\
+1 & \mbox{if } x\geq a+L.\\
\end{array}\right.$$
For negative values of $a$ with sufficiently large $|a|$, $\Omega$ is contained in the region $ x\geq a+L$ thus $\int_\Omega v_a(x,y) dX >0$, while, 
for sufficiently large and positive values of $a$,
$\int_\Omega v_a(x,y) dX <0$. Therefore, by continuity, there exists a value of $a$ for which $\int_\Omega v_a(x,y) dX =0$. We choose this value
and estimate $\mu_1(\Omega)$ from above thanks to this test function $v_a$.  Denoting by $R_a$ the rectangle $(a,a+L)\times (0,w)$ 
and using the fact that $v_a$ is constant outside $R_a$, we have on the one-hand
$$\int_\Omega |\nabla v_a|^2 dX=\int_{\Omega\cap R_a} |\nabla v_a|^2 dX \leq \int_{R_a} |\nabla v_a|^2 dX=\frac{\pi^2 w^2}{A^2}\,\frac{A}{2}.$$
On the other hand
\begin{align*}
\int_\Omega v_a^2 dX&=\int_{R_a} v_a^2 dX - \int_{R_a\setminus \Omega} v_a^2 dX + \int_{\Omega\setminus R_a} v_a^2 dX\\
                                      &\leq \int_{R_a} v_a^2 dX - |R_a\setminus \Omega| + |\Omega \setminus R_a| = \frac{A}{2}.
\end{align*}
where we used the fact that $|R_a\setminus \Omega| = |\Omega \setminus R_a|$ because $\Omega$ and $R_a$ have the same area. The estimate 
\eqref{HLLineq} follows.
\end{proof}

\medskip
We come back to the proof of the upper bound. We combine \eqref{HLLineq} with the inequality
$$\lambda_1(\Omega)\leq \frac{j_{01}^2}{r^2}\leq \frac{9j_{01}^2}{\omega^2},$$
the first inequality comes from the inclusion of the incircle into $\Omega$ and the second geometric inequality, saturated by the equilateral triangle,
can be found in the classical book of convex geometry \cite{YBb}.
Thus, we infer
$$F(\Omega)=|\Omega|^2\lambda_1(\Omega)\mu_1(\Omega)\leq \frac{\pi^2\omega^2\times 9j_{01}^2}{\omega^2}=9\pi^2j_{01}^2.$$
\end{proof}

The following inclusion is a direct corollary of Theorem \ref{theobounds}. 
\begin{corollary}
The set $\mathcal{E}^C$ is contained in the region $$\left\{(x,y)|\ x\ge |B|\lambda_1(B)\ \text{and} \ \frac{\pi^4}{4x}< y\leq  \min\left(|B|\mu_1(B),\frac{9\pi^2j_{01}^2}{x}\right) \right\},$$
where $B$ is any disc of $\mathbb{R}^2$.
\end{corollary}

\medskip
Let us finish this part by a remark on the behavior of the diagram near the point $A$ corresponding to discs.
\begin{proposition}
The diagram $\mathcal{E}^C$ has a vertical tangent at the point $A=(x(B),y(B))$.
\end{proposition}
To see this property, it suffices to consider a sequence $\Omega_\varepsilon$ of convex domains, converging to the unit disk $\mathbb{D}$ and such that the ratio
$$\dfrac{\pi \mu_1(\mathbb{D})-|\Omega_\varepsilon|\mu_1(\Omega_\varepsilon)}{|\Omega_\varepsilon|\lambda_1(\Omega_\varepsilon) - \pi \lambda_1(\mathbb{D})} \to +\infty.$$
For that purpose, the sequence of ellipses, centered at the origin and with semi-axis of length $1+\varepsilon$ and $1$ is convenient.
Indeed, it is known, see for example { \cite[Section 7.2.6]{BdP17} and \cite[Remark 7.39]{BdP17}} that we have the following estimates
$$|\Omega_\varepsilon|\lambda_1(\Omega_\varepsilon) - \pi \lambda_1(\mathbb{D})\leq C_1 \varepsilon^2$$
and 
$$\pi \mu_1(\mathbb{D})-|\Omega_\varepsilon|\mu_1(\Omega_\varepsilon) \geq {C_2}{\varepsilon},$$
for some positive constants $C_1$ and $C_2$. The result follows.

\subsection{Some numerical experiments and conjectures}\label{S:numerics}

In this section, we use numerical experiments to obtain more information on the diagram $\mathcal{E}^C$ and state some interesting conjectures. 

We want to provide a numerical approximation of the diagram $\mathcal{E}^C$. To do so, a natural idea is to generate a large number of convex sets (more precisely polygons) for each we compute the first Dirichlet and Neumann eigenvalues via a classical finite element method. Nevertheless, the task of (properly) generating random convex polygons is quite challenging and interesting on its own. The main difficulty is that one wants to design an efficient and fast algorithm that allows to obtain a uniform distribution of the generated random convex polygons. For more clarification, let us discuss two different (naive) approaches:
\begin{itemize}
    \item one easy way to generate random convex polygons is by rejection sampling. We generate a random set of points in a square; if they form a convex polygon, we return it, otherwise we try again. Unfortunately, the probability of a set of $n$ points uniformly generated inside a given square to be in convex position is equal to $p_n = \left(\frac{\binom{2n-2}{n-1}}{n!}\right)^2$, see \cite{random_polygon}. Thus, the random variable $X_n$ corresponding to the expected number of iterations needed to obtain a convex distribution follows a geometric law of parameter $p_n$, which means that its expectation is given by $\mathbb{E}(X_n)=\frac{1}{p_n}=\left(\frac{n!}{\binom{2n-2}{n-1}}\right)^2 $. For example, if $N=20$, the expected number of iterations is approximately equal to $2.10^9$, and since one iteration is performed in an average of $0.7$ seconds, this means that the algorithm will need about $50$ years to provide one convex polygon with $20$ sides. 
    \item Another natural approach is to generate random points and take their convex hull. This method is quite fast, as one can compute the convex hull of $N$ points in a $\mathcal{O}(N\log(N))$ time (see \cite{MR475616} for example), but it is not quite relevant since it yields to a biased distribution.   
\end{itemize}

In order to avoid the issues stated above, we use an algorithm presented in \cite{sander}, that is based on the work of P. Valtr \cite{random_polygon}, where the author computes 
the probability of a set of $n$ points uniformly generated inside a given square to be in convex position and remarks (in Section 4) that the proof yields a fast and 
non-biased method to generate random convex sets inside a given square. We also refer to \cite{sander} for a nice description of the steps of the method and a beautiful 
animation where one can follow each step, one also will find an implementation of Valtr's algorithm in Java that we decided to translate in Matlab to use in the  \texttt{PDEtool}.
To obtain the different figures below, we generate $10^5$ random convex polygons of unit area and number of sides between $3$ and $30$ for which we compute the 
first Dirichlet and Neumann eigenvalues, by a classical finite element method. The eigenvalues computations were performed using Matlab's toolbox for solving partial 
differential equations \texttt{PDEtool} on a personal computer The average time needed to compute one eigenvalue is approximately equal to $0.75$ second. We then obtain a 
cloud of dots that provides an approximation of the diagram $\mathcal{E}^C$. This approach has been used in several works, we refer for example 
to \cite{AH11}, \cite{FtJMAA} and \cite{FL21}.


It is always interesting to have information on the boundaries of Blaschke--Santal\'o diagrams and the extremal domains (which correspond to points of the boundaries). For the present $(\lambda_1,\mu_1)$-diagram $\mathcal{E}^C$, numerical experiments suggest the following conjecture: 
\begin{conjecture}\label{conj:1}
\begin{itemize}
    \item Except for the ball, the domains laying on the upper boundary of $\mathcal{E}^C$ are polygonal (i.e. the solution of the problem $\max\{\mu_1(\Omega), \Omega \mbox{ convex }, \lambda_1(\Omega)=x_0, |\Omega|=1\}$, where $x_0>\lambda_1(B) $, is a polygon). 
    \item The regular polygons are located on the upper boundary of the diagram $\mathcal{E}^C$.
    \item We denote by $T_e$ the equilateral triangle of unit area. There exists $x_0>\lambda_1(T_e)$ such that: 
    \begin{itemize}
        \item if $\ell \in [\lambda_1(T_e),x_0)$, then the solution of the problem $\max\{\mu_1(\Omega), \Omega \mbox{ convex }, \lambda_1(\Omega)=\ell, |\Omega|=1\}$ is given by a superequilateral triangle (which is isosceles with aperture greater that $\pi/3$). 
        \item If $\ell \in (x_0,+\infty)$, then the solution of  the problem $\max\{\mu_1(\Omega), \Omega \mbox{ convex }, \lambda_1(\Omega)=\ell, |\Omega|=1\}$ is given by a rectangle. 
        \item The solution of the problem $\max\{\mu_1(\Omega), \Omega \mbox{ convex }, \lambda_1(\Omega)=x_0, |\Omega|=1\}$ is given by both a rectangle and a superequilateral triangle. 
    \end{itemize}
\end{itemize}
\end{conjecture}


Let us now focus on the inequalities of Theorem \ref{theobounds}. It is interesting to visualise those inequalities in the Blaschke--Santal\'o diagram that helps us to guess the optimal bounds of the scaling invariant functional $F(\Omega)=|\Omega|^2\lambda_1(\Omega)\mu_1(\Omega)$ introduced in Theorem \ref{theobounds}. It is clear that inequalities of the type $F(\Omega)\leq c_0$ (or $F(\Omega)\ge c_0$) can be read in the diagram as the curve of a hyperbola $x\longmapsto \frac{c_0}{x}$ that will delimit a region which contains $\mathcal{E}^C$.
\begin{itemize}
    \item The lowest hyperbola that we managed to draw above the diagram corresponds to the choice of $c_0=F(B)$ (where $B$ is a ball), which suggests that among bounded planar convex sets, $F$ is maximized by balls, this provides a sharp upper bound, see Conjecture \ref{conj:2} and Figure \ref{fig:conjecture}. 
    \item As for the lower bound,   
numerical evidence suggests that for thin domains, rhombi are located in the lower part of the diagram $\mathcal{E}^C$, this suggest that the infimum of $F$ is asymptotically attained by vanishing thin rhombi and seems to be given by $\pi^2 j_{0,1}^2$, indeed, if we denote by $R_d$ a rhombus of unit area and diameter equal to $d>0$, we have $\mu_1(R_d)\underset{d\rightarrow +\infty}{\sim} \frac{4 j_{0,1}^2}{d^2}$ and $\lambda_1(R_d) \underset{d\rightarrow +\infty}{\sim} \frac{\pi^2 d^2}{4}$ (see the case of equality in \cite[Theorem 1]{MR1486739} for the first equivalence and \cite{MR2558181} for the second one). 
\end{itemize}
\vspace{2mm}

We summarize the above discussion in the following conjecture:
\begin{conjecture}\label{conj:2}
$$\forall \Omega {\in \mathcal{K}},\ \ \ \ \ \pi^2 j_{0,1}^2 <  F(\Omega) \leq F(B)= \pi^2j_{0,1}^2 {j'}_{1,1}^2,  $$
where $B$ is any disk of $\mathbb{R}^2$. The upper bound is an equality only for balls and the lower bound is asymptotically reached by any family of thin vanishing rhombi. 
\end{conjecture}

\begin{figure}[!h]
  \centering
    \includegraphics[scale=0.44]{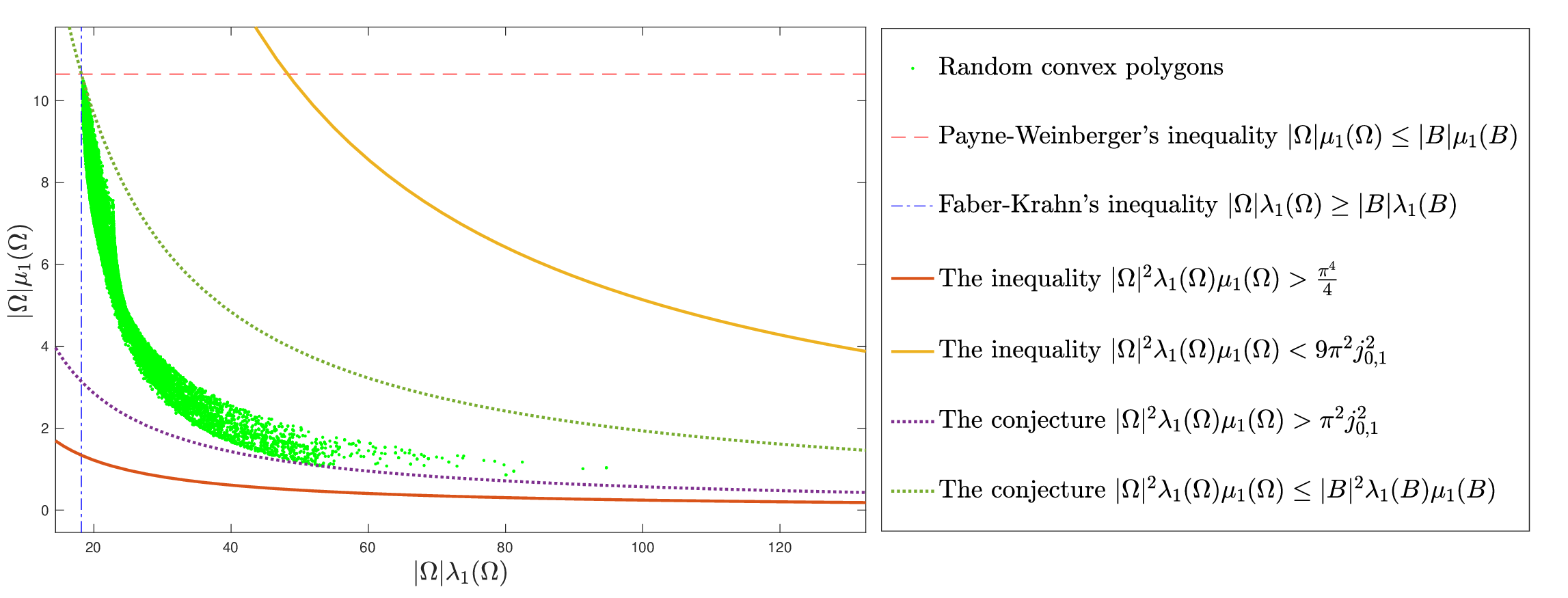}
    \caption{Graphical visualization of the inequalities of Theorem \ref{theobounds} and the conjectures of Conjecture \ref{conj:2}}
    \label{fig:conjecture}
\end{figure}  

The following existence result supports the claim of Conjecture \ref{conj:2}. 
\begin{proposition}
 There exists $\Omega^*\in \mathcal{K}$, such that: 
$$F(\Omega^*)=\max\limits_{\Omega\in \mathcal{K}} F(\Omega).$$
\end{proposition}
\begin{proof}
The proof follows the classical method of calculus of variations. Let $(\Omega_k)$ be a maximizing sequence of elements of $\mathcal{K}$ of unit area (i.e. such that $|\Omega_k|=1$ for every $k\in \mathbb N$ and $\lim\limits_{k\rightarrow +\infty} F(\Omega_k) = \sup\limits_{\Omega\in \mathcal{K}} F(\Omega)$). 

Let us assume that the diameter of $\Omega_k$ denoted by $D_k$ is such that $D_k\underset{k\rightarrow +\infty}{\longrightarrow} +\infty$.
Let us prove in this case  a geometric property,  namely
\begin{equation}\label{geomp}
\lim\limits_{D_k\rightarrow +\infty}\frac{w_k}{r_k}=2.
\end{equation}
here $w_k$ is the minimal width of $\Omega_k$ and $r_k$ its inradius.
Using the classical inequality, see \cite{YBb}
$$w_k^2\leq \sqrt{3} |\Omega_k| = \sqrt{3},$$
we infer 
$w_k\leq 3^{1/4}$.
Now, from the inequality see \cite{Scott79})
$$\sqrt{3}(\frac{w}{r}-2)D \leq 2w.$$ 
we see that $w_k/r_k \to 2$ as soon as $D_k\to +\infty$ (because the right-hand side is bounded).
Coming back to the eigenvalues, we have: 
$$\lambda_1(\Omega_k)\underset{k\rightarrow +\infty}{\sim} \frac{\pi^2}{4}\times\frac{1}{r_k^2},$$
see the proof of \cite[Proposition 5.1]{FtJMAA}. Thus, by combining this equivalence with the inequality \eqref{HLLineq}, we obtain:
\begin{equation}\label{eq:absurde}
\limsup\limits_{k\rightarrow +\infty} F(\Omega_k)\leq \limsup\limits_{k\rightarrow +\infty} 
\frac{\pi^4}{4}\times \left(\frac{w_k}{r_k}\right)^2 =\pi^4 <\pi^2 {j'}_{1,1}^2j_{0,1}^2= F(B).    
\end{equation}
Thus, the result of \eqref{eq:absurde} is a contradiction with the assumption that $(\Omega_k)$ is a maximizing sequence, which proves that (up to translations) there exists a bounded box $D\subset \mathbb{R}^2$ that contains all the $\Omega_k$. Thus, by Blaschke selection theorem {(see for example \cite[Theorem 1.8.7]{schneider})}, there exists a convex set $\Omega^*$ such that $(\Omega_k)$ converges to $\Omega^*$ (up to a subsequence) for the Hausdorff distance. By continuity of the area and the Dirichlet and Neumann eigenvalues for the Hausdorff distance in convex sets, we deduce that: 
$$F(\Omega^*)=\lim\limits_{k\rightarrow +\infty} F(\Omega_k) = \sup_{\Omega\in \mathcal{K}}F(\Omega).$$
This ends the proof. 
\end{proof}

%

\medskip\noindent

{\bf Acknowledgements}: 
This work was partially supported by the project ANR-18-CE40-0013 SHAPO financed by the French Agence Nationale de la Recherche (ANR).

\bibliographystyle{abbrv}
\bibliography{biblio.bib}

\end{document}